\documentclass[11pt]{amsart}
\usepackage[T1]{fontenc}
\usepackage[utf8]{inputenc}
\usepackage{amsmath,amssymb,amsthm}
\usepackage{enumitem}
\usepackage[final,expansion=false]{microtype}
\usepackage{hyperref}
\hypersetup{colorlinks=true,linkcolor=blue,citecolor=blue,urlcolor=blue}

\newtheorem{theorem}{Theorem}[section]
\newtheorem{lemma}[theorem]{Lemma}
\newtheorem{proposition}[theorem]{Proposition}
\newtheorem{corollary}[theorem]{Corollary}
\theoremstyle{definition}
\newtheorem{definition}[theorem]{Definition}
\newtheorem{question}[theorem]{Question}
\theoremstyle{remark}

\newcommand{\Katetov}{Kat\v{e}tov}
\newcommand{\conv}{\mathsf{conv}}
\newcommand{\Conv}{\mathsf{Conv}}
\newcommand{\Kle}{\le_K}

\title[Dense-set dependence in the \Katetov\ order]
{Dense-set dependence in the \Katetov\ order for uncountable coordinate ideals}

\author{Xing-Yu Hu, Zhang-Yi Luo}
\address{\begin{tabular}[t]{@{}l@{}}School of Mathematics and Statistics, Hanjiang Normal University\\ Shiyan, Hubei, China\end{tabular}}
\email{huxingyu@hjnu.com; lzy20050813@qq.com}

\subjclass[2020]{Primary 03E17, 54A20; Secondary 03E05, 03E15}
\keywords{\Katetov\ order, critical ideal, coordinate ideal, countable dense set, countable compact ordinal space, Continuum Hypothesis}
\date{\today}

\begin{document}

\begin{abstract}
For each countable ordinal \(\alpha\ge 2\), Filip\'ow, Kowalczuk and Kwela introduced an ideal \(\conv_\alpha\) on the countable compact ordinal space \(\omega^\alpha+1\). Kowalczuk later proved that, for each countable limit ordinal \(\lambda\), the ideal \(\conv_{<\lambda}\) is the greatest lower bound of \(\{\conv_\beta:\beta<\lambda\}\) in the \Katetov\ order.

At the first uncountable level, let \(A\subseteq[2,\omega_1)\) be uncountable and let \(D\) be a countable dense subset of \(X_A=\prod_{\alpha\in A}(\omega^\alpha+1)\). The coordinate ideal \(\Conv(A,D)\) on \(D\) consists of those \(B\subseteq D\) with \(\pi_\alpha[B]\in\conv_\alpha\) for every \(\alpha\in A\). For a pair \(D\subseteq E\) of countable dense sets, call \(\alpha\) \emph{non-small} if \(\pi_\alpha[E\setminus D]\notin\conv_\alpha\). In ZFC, if at most countably many coordinates are non-small, then \(\Conv(A,D)\equiv_K\Conv(A,E)\). Under CH this countability bound is sharp: for every \(A\subseteq[3,\omega_1)\) with \(|A|=\aleph_1\), there are countable dense sets \(D\subseteq D^*\subseteq X_A\) such that
\[
        \Conv(A,D^*)\le_K\Conv(A,D)
        \quad\text{but}\quad
        \Conv(A,D)\not\le_K\Conv(A,D^*),
\]
and in particular \(\Conv(A,D)\) and \(\Conv(A,D^*)\) are not \Katetov\ equivalent. The non-reduction is obtained, under CH, by diagonalizing along \(\omega_1\) coordinates against the elements of \(\omega^\omega\) that code retractions \(D^*\to D\).
\end{abstract}

\maketitle

\section{Introduction}

The ideals \(\conv_\alpha\), for countable \(\alpha\ge2\), were introduced by Filip\'ow, Kowalczuk and Kwela as part of their characterization of countable compact spaces homeomorphic to \(\omega^\alpha\cdot n+1\) \cite{FilipowKowalczukKwela2025}; they are defined on the ordinal compacta \(\omega^\alpha+1\) (see Section~\ref{sec:prelim} for the precise definition). The ordinal models used here go back to the classical analysis of countable compact spaces by Mazurkiewicz and Sierpi\'nski \cite{MazurkiewiczSierpinski1920}. For background on ideals, definable ideals, and the \Katetov\ order we refer to \cite{Katetov1968,Solecki1999,Farah2000,Hrusak2011,Hrusak2017,BarbarskiFilipowMrozekSzuca2013}.

The \Katetov\ order on the family \(\{\conv_\alpha:2\le\alpha<\omega_1\}\) is anti-monotone in the ordinal index: for every \(\alpha<\beta<\omega_1\),
\[
        \conv_\beta\Kle\conv_\alpha,
        \qquad
        \conv_\alpha\not\Kle\conv_\beta;
\]
see \cite{FilipowKowalczukKwela2025,Kowalczuk2026}. For a countable limit ordinal \(\lambda\), the ideal \(\conv_\lambda\) is \emph{not} the greatest lower bound of \(\{\conv_\beta:\beta<\lambda\}\); Kowalczuk introduced an auxiliary ideal \(\conv_{<\lambda}\) and proved that it is the greatest lower bound of this family in the \Katetov\ order \cite[Proposition~3.2]{Kowalczuk2026}.

The first uncountable analogue starts with an uncountable \(A\subseteq[2,\omega_1)\). Set
\[
        X_A=\prod_{\alpha\in A}(\omega^\alpha+1).
\]
The product \(X_A\) is separable by the Hewitt--Marczewski--Pondiczery theorem \cite[Theorem~2.3.15]{Engelking1989}, since \(|A|\le\aleph_1\le 2^{\aleph_0}\). For a countable dense \(D\subseteq X_A\), define
\[
        \Conv(A,D)=
        \{B\subseteq D:\pi_\alpha[B]\in\conv_\alpha
        \text{ for every }\alpha\in A\}.
\]
The definition of \(\Conv(A,D)\) depends on the choice of \(D\). The question is whether this dependence is genuine, namely whether two countable dense sets can give non-\Katetov-equivalent ideals.

This dependence is not a formal nuisance. A countable dense set in the product projects densely to each coordinate, but it also fixes how the countably many chosen points are aligned across the coordinates. The ideal \(\Conv(A,D)\) tests smallness coordinate by coordinate, whereas a \Katetov\ witness is a single map between the underlying countable sets. Thus a change of dense set can be invisible at many individual coordinates and still affect the existence of a global witness.

There is no conflict with the invariance, in topological representations of ideals, under the choice of a countable dense set. In that context, a single \(\sigma\)-ideal on a separable metrizable space gives isomorphic ideals on any two countable dense subsets \cite[Proposition~2.1]{KwelaSabok2015}. In the present construction, by contrast, the space is an uncountable product and smallness is imposed separately at the coordinates by the ideals \(\conv_\alpha\). Thus the comparison made here is one between ideals on countable sets via \Katetov\ reducibility, rather than a representation theorem for a fixed \(\sigma\)-ideal \cite{BarbarskiFilipowMrozekSzuca2013,Hrusak2017}.

There is an elementary asymmetry in this question. If \(D\subseteq E\), then the inclusion of \(D\) into \(E\) always gives
\[
        \Conv(A,E)\Kle\Conv(A,D).
\]
Thus the only possible obstruction is the reverse inequality \(\Conv(A,D)\Kle\Conv(A,E)\). The positive and negative results below both focus on this direction. In the positive result, the new points of \(E\setminus D\) are small at all but countably many coordinates. In the negative result, CH is used to distribute the possible reverse witnesses over \(\omega_1\) coordinates and to defeat them one at a time.

For countable dense sets \(D\subseteq E\subseteq X_A\), call \(\alpha\in A\) a \emph{non-small coordinate} if \(\pi_\alpha[E\setminus D]\notin\conv_\alpha\). A ZFC absorption argument shows that countably many non-small coordinates can always be absorbed:
\[
        \bigl|\{\alpha\in A:\pi_\alpha[E\setminus D]\notin\conv_\alpha\}\bigr|\le\aleph_0
        \quad\Longrightarrow\quad
        \Conv(A,D)\equiv_K\Conv(A,E)
\]
(Proposition~\ref{prop:countably-active-absorbed}). Under CH this bound is sharp.

\begin{theorem}\label{thm:main}
Assume CH. Let \(A\subseteq[3,\omega_1)\) have cardinality \(\aleph_1\). Then there are countable dense sets
\[
        D\subseteq D^*\subseteq X_A
\]
such that
\[
        \Conv(A,D^*)\le_K\Conv(A,D)
        \quad\text{but}\quad
        \Conv(A,D)\not\le_K\Conv(A,D^*).
\]
In particular, \(\Conv(A,D)\) and \(\Conv(A,D^*)\) are not \Katetov\ equivalent.
\end{theorem}

The restriction to \(A\subseteq[3,\omega_1)\) is required by the auxiliary sets \(L_\beta\subseteq\omega^\beta+1\) of Section~\ref{sec:L-beta}, whose definition uses ordinals of the form \(\omega^2\cdot n+\omega(k+1)\) and therefore requires \(\beta\ge3\); the case \(\alpha=2\) is not covered by our argument.

This theorem is not a solution of the greatest-lower-bound problem at \(\omega_1\). It shows instead that a natural uncountable-coordinate construction is not canonical unless the dense set is part of the data. The point is that, beyond the countable limit levels treated by \(\conv_{<\lambda}\), the product construction introduces a second parameter, namely the countable dense subset of the product. The theorem says that this parameter cannot be suppressed in general.

The obstruction in Theorem~\ref{thm:main} comes from coding retractions \(D^*\to D\) by elements of \(\omega^\omega\). Once \(D^*\setminus D\) and \(D\) are enumerated, each such retraction determines a code. Under CH these codes can be assigned to the coordinates of \(A\). At the coordinate assigned to a code, the construction produces a test set \(B(J_\xi)\) which is small for \(\Conv(A,D)\) but whose inverse image under the corresponding retraction is not small for \(\Conv(A,D^*)\). The reservoir-selection construction supplies the off-diagonal smallness needed for this test.

Theorem~\ref{thm:main} therefore shows that the ideal \(\Conv(A,D)\) may depend, up to \Katetov\ equivalence, on the chosen dense set \(D\). The greatest-lower-bound problem for \(\{\conv_\alpha:2\le\alpha<\omega_1\}\) remains open as Question~\ref{q:GLB}.

The positive result rules out counterexamples with only countably many active coordinates. If the added points are non-small at only countably many coordinates, the reverse witness can be built by matching finitely many of those coordinates at a time. The example under CH prevents this finite matching argument from being extended to \(\omega_1\) many active coordinates: every possible code for a retraction is assigned a coordinate at which its inverse image of a small test set becomes non-small.

\section{Preliminaries}\label{sec:prelim}

We use standard terminology from general topology as in \cite{Engelking1989}, and standard terminology for ideals and the \Katetov\ order as in \cite{Katetov1968,Hrusak2011,Hrusak2017}. All ordinal spaces carry the order topology, and all products carry the product topology.

If \(\mathcal I\) is an ideal on \(X\) and \(\mathcal J\) is an ideal on \(Y\), we write
\[
        \mathcal I\Kle\mathcal J
\]
if there is a map \(f\colon Y\to X\) such that \(B\in\mathcal I\) implies \(f^{-1}[B]\in\mathcal J\). We write \(\mathcal I\equiv_K\mathcal J\) if \(\mathcal I\Kle\mathcal J\) and \(\mathcal J\Kle\mathcal I\).

Let \(X\) be a topological space and let \(D\subseteq X\) be countable. The ideal \(\conv_X(D)\) on \(D\) consists of all subsets of \(D\) contained in a finite union of ranges of sequences from \(D\) that converge in \(X\). When \(X=\omega^\alpha+1\) (with the order topology) and \(D=X\), we write
\[
        \conv_\alpha=\conv_{\omega^\alpha+1}(\omega^\alpha+1).
\]

\begin{lemma}\label{lem:conv-derived-set}
For every countable \(\alpha\ge2\),
\[
        \conv_\alpha=\{B\subseteq\omega^\alpha+1:B^d\text{ is finite}\},
\]
where \(B^d\) denotes the derived set of \(B\) in \(\omega^\alpha+1\).
\end{lemma}

\begin{proof}
The space \(\omega^\alpha+1\) is compact and metrizable. If \(B\) is contained in a finite union of ranges of sequences converging in \(\omega^\alpha+1\), then each such range has at most one accumulation point, so \(B^d\) is finite. Conversely, suppose \(B^d=\{x_1,\dots,x_r\}\) is finite. Choose open neighbourhoods \(U_i\ni x_i\) whose closures are pairwise disjoint. The set \(B\setminus\bigcup_{i\le r}U_i\) has no accumulation point in the compact space \(\omega^\alpha+1\), hence is finite. For each \(i\), every accumulation point of \(B\cap U_i\) lies in \(B^d\cap\overline{U_i}=\{x_i\}\). Hence every neighbourhood of \(x_i\) contains all but finitely many points of \(B\cap U_i\). Thus, if \(B\cap U_i\) is infinite, any one-to-one enumeration of it converges to \(x_i\); if it is finite, it is already a finite union of ranges of constant sequences. Thus \(B\) is a finite union of ranges of convergent sequences together with a finite set, whence \(B\in\conv_\alpha\).
\end{proof}

Lemma~\ref{lem:conv-derived-set} is a standard equivalence for countable compact metric ordinal spaces; see \cite[Section~3]{FilipowKowalczukKwela2025}. Below we use this derived-set description of \(\conv_\alpha\). Thus every finite set, and the range of every convergent sequence, belongs to \(\conv_\alpha\), and adjoining finitely many points to a member of \(\conv_\alpha\) leaves it in the ideal.

For a countable limit ordinal \(\lambda\), let
\[
        \conv_{<\lambda}=\{B\subseteq\omega^\lambda+1:
        B\cap\omega^\beta\in\conv_\beta\text{ for every }\beta<\lambda\}.
\]
Kowalczuk proved that \(\conv_{<\lambda}\) is the greatest lower bound of \(\{\conv_\beta:\beta<\lambda\}\) in the \Katetov\ order \cite[Proposition~3.2]{Kowalczuk2026}.

For the lower-bound problem, the particular unbounded set \(A\) is not essential. By the anti-monotonicity \(\alpha<\beta\Rightarrow\conv_\beta\Kle\conv_\alpha\), the transitivity of \(\Kle\), and Kowalczuk's theorem quoted above, every non-empty \(A\subseteq[2,\omega_1)\) falls into one of three cases. If \(A\) has maximum \(\lambda\), then \(\conv_\lambda\) is the greatest lower bound of \(\{\conv_\alpha:\alpha\in A\}\). If \(\lambda=\sup A<\omega_1\) and \(A\) has no maximum, then \(\lambda\) is a countable limit ordinal and \(\conv_{<\lambda}\) is the greatest lower bound of \(\{\conv_\alpha:\alpha\in A\}\). If \(A\) is unbounded in \(\omega_1\), then \(\{\conv_\alpha:\alpha\in A\}\) and \(\{\conv_\alpha:2\le\alpha<\omega_1\}\) have the same lower bounds in the \Katetov\ order. Since every uncountable \(A\subseteq[2,\omega_1)\) is unbounded in \(\omega_1\), the last case applies to the sets \(A\) used in Theorem~\ref{thm:main}.

Fix an uncountable \(A\subseteq[3,\omega_1)\). For every countable dense \(D\subseteq X_A\), the set
\[
        \Conv(A,D)=\{B\subseteq D:\pi_\alpha[B]\in\conv_\alpha
        \text{ for every }\alpha\in A\}
\]
is a proper ideal on \(D\): heredity and closure under finite unions follow coordinatewise from the corresponding properties of each \(\conv_\alpha\), and properness follows from the fact that, for every \(\alpha\in A\), the projection \(\pi_\alpha[D]\) is dense in \(\omega^\alpha+1\) (hence has infinite derived set) and therefore does not belong to \(\conv_\alpha\).

Throughout the paper, comparisons between such ideals are made inside the same ambient product \(X_A\). The dense-set parameter is therefore the only object that changes. This convention is used in the absorption result, where \(D\) is enlarged by a countable set, and in the CH construction, where \(D^*\) is obtained from \(D\) by adding a second countable family of points.

\section{ZFC stability for countable non-small support}\label{sec:zfc-stability}

We first prove a ZFC stability fact for enlargements whose non-small support is countable. In this case the new points can be pushed back into the old dense set by a \Katetov\ witness. The proof is stated for an arbitrary countable set \(E\), not only for a dense one, because only the added part \(E\setminus D\) is used.

The countability assumption is used in a finite-support way. After the non-small coordinates are enumerated, the point of \(D\) chosen to replace the \(n\)-th new point is required to match only the first finitely many active coordinates. For each fixed coordinate, all but finitely many new points are then controlled by the corresponding safe map; the remaining finitely many points are harmless for \(\conv_\alpha\). This is the reason the argument stops exactly at countable non-small support.

Let \(\beta\ge2\), and let \(P\subseteq\omega^\beta+1\). A map \(\theta\colon P\to\omega^\beta+1\) is called \(\conv_\beta\)-\emph{safe} if
\[
        S\in\conv_\beta
        \quad\Longrightarrow\quad
        \theta^{-1}[S]\in\conv_\beta
\]
for every \(S\subseteq\omega^\beta+1\).

\begin{lemma}
\label{lem:safe-approximation-isolated}
Let \(\beta\ge2\) be a countable ordinal. Let \(P\subseteq\omega^\beta+1\) be countable, and let \(Q\) be a dense subset of \(\omega^\beta+1\) consisting of isolated points. Then there is a \(\conv_\beta\)-safe map \(\theta\colon P\to Q\).
\end{lemma}

\begin{proof}
If \(P\) is finite (in particular if \(P=\varnothing\)), any map \(P\to Q\) is \(\conv_\beta\)-safe because finite sets are \(\conv_\beta\)-small. Assume \(P\) is infinite, and fix a compatible metric \(d\) on \(\omega^\beta+1\). Enumerate \(P\) without repetitions as \(\{p_n:n<\omega\}\), and choose \(\theta(p_n)\in Q\) recursively so that
\[
        d(\theta(p_n),p_n)<2^{-n}.
\]
If \(p_n\) is isolated, then \(\{p_n\}\) is a non-empty open set, so by density \(p_n\in Q\); set \(\theta(p_n)=p_n\). If \(p_n\) is non-isolated, choose \(\theta(p_n)\in Q\) inside the open ball of radius \(2^{-n}\) around \(p_n\), avoiding all points of \(Q\) chosen at earlier non-isolated stages. This is possible because \(Q\cap U\) is infinite for every neighbourhood \(U\) of such a point. Indeed, after deleting finitely many earlier \(Q\)-choices, the remaining open set still contains the non-isolated point and therefore meets \(Q\) by density.

Each fibre of \(\theta\) is finite, since any \(q\in Q\) is chosen at most once at a non-isolated stage and at most once at an isolated stage. We verify safety by contraposition: let \(T\subseteq P\) with \(T\notin\conv_\beta\). By Lemma~\ref{lem:conv-derived-set}, \(T^d\) is infinite. For \(\lambda\in T^d\), choose distinct \(p_{n_k}\in T\) with \(p_{n_k}\to\lambda\); passing to a subsequence, we may assume \(n_k\to\infty\). Then
\[
        d(\theta(p_{n_k}),\lambda)\le d(\theta(p_{n_k}),p_{n_k})+d(p_{n_k},\lambda)\to 0,
\]
so \(\theta(p_{n_k})\to\lambda\). Since the fibres of \(\theta\) are finite, a further subsequence has pairwise distinct \(\theta(p_{n_k})\), so \(\lambda\in(\theta[T])^d\). Hence \(T^d\subseteq(\theta[T])^d\), and \(\theta[T]\notin\conv_\beta\).

Now let \(S\in\conv_\beta\) and set \(T=\theta^{-1}[S]\). Since \(\theta[T]\subseteq S\in\conv_\beta\), we have \(\theta[T]\in\conv_\beta\); by the contrapositive just proved, \(T\in\conv_\beta\). Thus \(\theta\) is \(\conv_\beta\)-safe.
\end{proof}

\begin{proposition}
\label{prop:countably-active-absorbed}
Let \(A\subseteq[2,\omega_1)\) be non-empty, let \(D\) be a countable dense subset of \(X_A\), and let \(E\subseteq X_A\) be countable. Put
\[
        C(E,D)=\{\alpha\in A:\pi_\alpha[E\setminus D]\notin\conv_\alpha\}.
\]
If \(C(E,D)\) is countable, then \(\Conv(A,D)\equiv_K\Conv(A,D\cup E)\).
\end{proposition}

\begin{proof}
Put \(E_0=E\setminus D\) and \(D^*=D\cup E\). The inequality \(\Conv(A,D^*)\Kle\Conv(A,D)\) is witnessed by the inclusion \(D\hookrightarrow D^*\): for \(B\in\Conv(A,D^*)\) and every \(\alpha\in A\), \(\pi_\alpha[B\cap D]\subseteq\pi_\alpha[B]\in\conv_\alpha\), so \(B\cap D\in\Conv(A,D)\) by heredity of each \(\conv_\alpha\). It remains to construct a witness for \(\Conv(A,D)\Kle\Conv(A,D^*)\).

If \(E_0\) is finite, fix \(d_0\in D\) and define \(r\colon D^*\to D\) by \(r|_D=\mathrm{id}_D\) and \(r(x)=d_0\) for \(x\in E_0\). For \(B\in\Conv(A,D)\) the set \(r^{-1}[B]\) is contained in \(B\cup E_0\), and adjoining a finite set to a \(\conv_\alpha\)-small set keeps it small; hence \(r\) is a \Katetov\ witness.

Assume from now on that \(E_0=\{e_n:n<\omega\}\) is enumerated without repetitions. For each \(\alpha\in C(E,D)\), apply Lemma~\ref{lem:safe-approximation-isolated} to \(P_\alpha=\pi_\alpha[E_0]\) and to the set of isolated points of \(\omega^\alpha+1\) (which is dense, since the isolated points of any countable ordinal space are dense). This produces a \(\conv_\alpha\)-safe map
\[
        \theta_\alpha\colon P_\alpha\to\omega^\alpha+1
\]
whose range consists of isolated points.

If \(C(E,D)\) is finite, enumerate it as \(\{\alpha_0,\ldots,\alpha_{M-1}\}\), and for each \(n<\omega\) use the density of \(D\) to choose \(d_n\in D\) with
\[
        d_n(\alpha_m)=\theta_{\alpha_m}(e_n(\alpha_m))
        \qquad(m<M).
\]
This is possible because each prescribed value is an isolated point of \(\omega^{\alpha_m}+1\), so the prescription cuts out a non-empty basic open subset of \(X_A\). If \(C(E,D)\) is countably infinite, enumerate it as \(\{\alpha_m:m<\omega\}\) and, by the same argument, choose \(d_n\in D\) with
\[
        d_n(\alpha_m)=\theta_{\alpha_m}(e_n(\alpha_m))
        \qquad(m\le n).
\]

Define \(r\colon D^*\to D\) by \(r|_D=\mathrm{id}_D\) and \(r(e_n)=d_n\). We check that \(r\) witnesses \(\Conv(A,D)\Kle\Conv(A,D^*)\). Let \(B\in\Conv(A,D)\), and fix \(\alpha\in A\). If \(\alpha\notin C(E,D)\), then
\[
        \pi_\alpha[r^{-1}[B]]\subseteq\pi_\alpha[B]\cup\pi_\alpha[E_0]\in\conv_\alpha,
\]
since both terms are \(\conv_\alpha\)-small. If \(\alpha\in C(E,D)\) and \(C(E,D)\) is finite, then any \(e_n\in E_0\cap r^{-1}[B]\) satisfies
\[
        e_n(\alpha)\in\theta_\alpha^{-1}[\pi_\alpha[B]]\in\conv_\alpha
\]
by safety of \(\theta_\alpha\), while \(r^{-1}[B]\cap D=D\cap B\) contributes only \(\pi_\alpha[B]\in\conv_\alpha\). If \(\alpha=\alpha_m\) and \(C(E,D)\) is countably infinite, the same conclusion holds for all \(e_n\) with \(n\ge m\), while the indices \(n<m\) contribute only the finite set \(\{e_n(\alpha):n<m\}\). In every case \(\pi_\alpha[r^{-1}[B]]\in\conv_\alpha\); since \(\alpha\) was arbitrary, \(r^{-1}[B]\in\Conv(A,D^*)\).
\end{proof}

\begin{corollary}
\label{cor:zfc-countable-barrier}
Let \(A\subseteq[2,\omega_1)\) be non-empty, and let \(D_0,D_1\) be countable dense subsets of \(X_A\). If both sets
\[
        \{\alpha\in A:\pi_\alpha[D_1\setminus D_0]\notin\conv_\alpha\},
        \qquad
        \{\alpha\in A:\pi_\alpha[D_0\setminus D_1]\notin\conv_\alpha\}
\]
are countable, then \(\Conv(A,D_0)\equiv_K\Conv(A,D_1)\). Consequently, any dense-set dependence example must have uncountably many non-small coordinates in at least one of the two directions \(D_1\setminus D_0\) and \(D_0\setminus D_1\).
\end{corollary}

\begin{proof}
Apply Proposition~\ref{prop:countably-active-absorbed} first to \(D_0\subseteq D_0\cup D_1\) and then to \(D_1\subseteq D_0\cup D_1\).
\end{proof}

\section{The sets \texorpdfstring{\(L_\beta\)}{L beta}}\label{sec:L-beta}

For \(\beta\ge3\), define
\[
        L_\beta=
        \{\omega^2\cdot n+\omega\cdot(k+1):1\le n<\omega,\ k<\omega\}.
\]
The underlying set is independent of \(\beta\). The subscript specifies the ambient ordinal compactum \(\omega^\beta+1\), and hence the topology and the ideal involved.

The set \(L_\beta\) is chosen so that it is not \(\conv_\beta\)-small while its complement remains dense. Thus the added points of \(D^*\setminus D\) can be placed on \(L_\beta\), while the old dense set \(D\) uses points outside \(L_\beta\) at the same coordinate.

The formula for \(L_\beta\) gives these two properties uniformly for all \(\beta\ge3\). The points of \(L_\beta\) are non-isolated, and the first parameter in \(\omega^2\cdot n+\omega(k+1)\) gives infinitely many distinct limit points. Hence \(L_\beta\) is non-small, while the isolated points, which are disjoint from \(L_\beta\), keep the complement dense enough for the later construction of \(D\).

\begin{lemma}\label{lem:L-beta}
For every countable \(\beta\ge3\), the set \(L_\beta\) is countable, consists of non-isolated points of \(\omega^\beta+1\), has dense complement in \(\omega^\beta+1\), and does not belong to \(\conv_\beta\).
\end{lemma}

\begin{proof}
Every point of \(L_\beta\) is an ordinal of the form \(\omega^2\cdot n+\omega\cdot(k+1)<\omega^3\le\omega^\beta\), so \(L_\beta\subseteq\omega^\beta+1\). Each such point is a limit ordinal (because \(\omega\cdot(k+1)\) is a limit ordinal) and is therefore non-isolated in \(\omega^\beta+1\). The isolated points are dense in every countable ordinal space and are disjoint from \(L_\beta\), so \(L_\beta\) has dense complement. For each \(n\ge1\), the sequence \(\{\omega^2\cdot n+\omega\cdot(k+1):k<\omega\}\) converges to \(\omega^2\cdot(n+1)\), so the derived set of \(L_\beta\) contains the infinite set \(\{\omega^2\cdot(n+1):1\le n<\omega\}\). By Lemma~\ref{lem:conv-derived-set}, \(L_\beta\notin\conv_\beta\).
\end{proof}

\begin{lemma}\label{lem:reservoir-splitting}
Let \(\mathcal K\) be a countable family of infinite subsets of \(\omega\), and let \(\varphi\colon\omega\to\omega\). There is \(R\subseteq\omega\) such that \(\varphi^{-1}[R]\) is infinite and \(K\setminus R\) is infinite for every \(K\in\mathcal K\).
\end{lemma}

\begin{proof}
Put \(V=\varphi[\omega]\). If \(V\) is finite, then since \(\omega=\bigcup_{j\in V}\varphi^{-1}(\{j\})\), some \(j\in V\) has infinite preimage; set \(R=\{j\}\). Then \(\varphi^{-1}[R]\) is infinite, and \(K\setminus R=K\setminus\{j\}\) is infinite for every infinite \(K\in\mathcal K\).

Assume \(V\) is infinite. Let
\[
        \mathcal H=\{K\cap V:K\in\mathcal K\text{ and }K\cap V
        \text{ is infinite}\}.
\]
If \(\mathcal H=\varnothing\), then \(K\cap V\) is finite for every \(K\in\mathcal K\), hence \(K\setminus V\) is infinite; take any infinite \(R\subseteq V\), so that \(K\setminus R\supseteq K\setminus V\) is infinite, and \(\varphi^{-1}[R]\) is infinite because \(R\subseteq V=\varphi[\omega]\) is infinite.

Otherwise, fix a surjection \(i\mapsto H_i\) from \(\omega\) onto \(\mathcal H\) such that every member of \(\mathcal H\) equals \(H_i\) for infinitely many \(i\); this is possible since \(\mathcal H\) is non-empty and at most countable. Recursively choose distinct \(a_i,b_i\in H_i\) avoiding all previously chosen points; this is possible since \(H_i\) is infinite and only finitely many points have been chosen at each stage. Put every \(a_i\) into \(R\) and every remaining point of \(V\) into \(R\), while keeping every \(b_i\) and every point outside \(V\) out of \(R\); thus \(R\subseteq V\) and \(R=V\setminus\{b_i:i<\omega\}\). Then \(R\) contains all \(a_i\), hence \(R\) is infinite and \(\varphi^{-1}[R]\) is infinite. For \(K\in\mathcal K\), either \(K\cap V\) is finite, in which case \(K\setminus R\supseteq K\setminus V\) is infinite; or \(K\cap V=H\) for some \(H\in\mathcal H\), in which case \(K\setminus R\supseteq\{b_i:H_i=H\}\), which is infinite because \(H=H_i\) for infinitely many \(i\) and the \(b_i\) are pairwise distinct.
\end{proof}

\begin{lemma}\label{lem:positive-indexing}
Let \(\beta\ge3\), and let \(M\subseteq\omega\) be infinite. There is an enumeration \(L_\beta=\{\ell_m:m<\omega\}\) without repetitions such that \(\{\ell_m:m\in M\}\notin\conv_\beta\).
\end{lemma}

\begin{proof}
Let \(L^{\mathrm{ev}}\) and \(L^{\mathrm{od}}\) be the parts of \(L_\beta\) obtained by requiring the first parameter \(n\) to be even and odd, respectively. By the argument of Lemma~\ref{lem:L-beta}, both pieces have infinite derived set, and hence both lie outside \(\conv_\beta\). Thus \(L_\beta=L^{\mathrm{ev}}\sqcup L^{\mathrm{od}}\).

If \(\omega\setminus M\) is infinite, fix a bijection of \(L^{\mathrm{ev}}\) with \(M\) and of \(L^{\mathrm{od}}\) with \(\omega\setminus M\), and let \(\ell_m\) be the corresponding image; then \(\{\ell_m:m\in M\}=L^{\mathrm{ev}}\notin\conv_\beta\). If \(\omega\setminus M\) is finite, choose a finite \(F\subseteq L_\beta\) with \(|F|=|\omega\setminus M|\), enumerate \(F\) along \(\omega\setminus M\) and \(L_\beta\setminus F\) along \(M\); since removing finitely many points from \(L_\beta\) leaves its derived set infinite, \(\{\ell_m:m\in M\}=L_\beta\setminus F\notin\conv_\beta\).
\end{proof}

\begin{definition}\label{def:reservoir-hmp}
Let \(I\) be an index set, and let \(S_i\subseteq\omega\) be infinite for each \(i\in I\). A family of maps \(\{h_i\colon S_i\to\omega:i\in I\}\) is \emph{reservoir-HMP independent} if for every finite \(F\subseteq I\) and every map \(s\colon F\to\omega\), the set
\[
        \Bigl\{j\in\bigcap_{i\in F}S_i:h_i(j)=s(i)\text{ for every }i\in F\Bigr\}
\]
is infinite.
\end{definition}

\section{Reservoir selection along \texorpdfstring{\(\omega_1\)}{omega1}}\label{sec:reservoir-selection}

The construction below produces the sets \(J_\xi\) needed for the diagonal argument. At stage \(\xi\), the set \(J_\xi\) must be large enough, through \(\varphi_\xi^{-1}[J_\xi]\), to give a non-small diagonal trace at the coordinate \(\beta_\xi\). At the same time, it must be small enough relative to all previously constructed coordinates so that later off-diagonal projections remain \(\conv\)-small. The auxiliary cells in the proof keep these two requirements compatible throughout the recursion.

No set-theoretic hypothesis is used in this proposition. The only reason the recursion can be carried out below \(\omega_1\) is that each initial segment \(\xi<\omega_1\) is countable. Thus, at stage \(\xi\), only countably many old finite-pattern requirements have to be preserved, and the compactness argument used to thin \(M_\xi^0\) takes place in a countable product of ordinal compacta.

\begin{proposition}
\label{prop:trace-controlled-selection}
Let \(\kappa\le\omega_1\). Let \(\Phi=\{\varphi_\xi:\xi<\kappa\}\subseteq\omega^\omega\), and let \(\{\beta_\xi:\xi<\kappa\}\) be pairwise distinct ordinals in \([3,\omega_1)\). For each \(\xi<\kappa\), fix a sequence
\[
        \{q_n^\xi:n<\omega\}\subseteq
        (\omega^{\beta_\xi}+1)\setminus L_{\beta_\xi}
\]
whose range is dense in \(\omega^{\beta_\xi}+1\). Then there exist sets
\[
        J_\xi\subseteq R_\xi\subseteq\omega,
        \qquad S_\xi=\omega\setminus R_\xi,
\]
maps \(h_\xi\colon S_\xi\to\omega\), and enumerations without repetitions
\[
        L_{\beta_\xi}=\{\ell_m^\xi:m<\omega\}
\]
satisfying \textup{(i)}--\textup{(iii)}.
\begin{enumerate}[label=\textup{(\roman*)}]
\item \(\{\ell_m^\xi:\varphi_\xi(m)\in J_\xi\}\notin\conv_{\beta_\xi}\) for every \(\xi<\kappa\).
\item The family \(\{h_\xi:\xi<\kappa\}\) is reservoir-HMP independent.
\item For every pair of distinct \(\eta,\xi<\kappa\),
\[
        \{q_n^\eta:n\in h_\eta[J_\xi\cap S_\eta]\}\in\conv_{\beta_\eta}.
\]
\end{enumerate}
\end{proposition}

\begin{proof}
We construct the required objects by recursion on \(\xi<\kappa\), maintaining the auxiliary invariant that whenever \(F,G\subseteq\xi\) are finite and \(s\colon F\to\omega\), the set
\[
        C(F,G,s)=
        \bigcap_{\eta\in F}\bigl(S_\eta\cap h_\eta^{-1}(\{s(\eta)\})\bigr)
        \setminus\bigcup_{\zeta\in G}J_\zeta
\]
is infinite, with the convention that the intersection over \(F=\varnothing\) is \(\omega\) and the union over \(G=\varnothing\) is empty. The case \(G=\varnothing\) of this invariant is the finite-pattern density condition needed for reservoir-HMP independence; the bookkeeping by \(G\) ensures that subsequent removals of sets \(J_\zeta\) still leave each old cell infinite. The invariant is automatic at limit stages \(\xi\): any finite \(F,G\subseteq\xi\) lies in some \(\xi_0<\xi\), and \(C(F,G,s)\) depends only on objects constructed at stages \(\le\xi_0\).

Assume the construction has been carried out below \(\xi\). Since \(\xi<\omega_1\) is countable, the family of finite subsets of \(\xi\) is countable, and hence so is the family
\[
        \mathcal C_\xi=\{C(F,G,s):F,G\subseteq\xi\text{ finite},\ s\colon F\to\omega\}.
\]
By the invariant every member of \(\mathcal C_\xi\) is infinite. Apply Lemma~\ref{lem:reservoir-splitting} to \(\mathcal C_\xi\) and \(\varphi_\xi\) to obtain \(R_\xi^0\subseteq\omega\) such that \(M_\xi^0=\varphi_\xi^{-1}[R_\xi^0]\) is infinite and \(C\setminus R_\xi^0\) is infinite for every \(C\in\mathcal C_\xi\). In particular, \(S_\xi:=\omega\setminus R_\xi^0\) meets every old cell in an infinite set.

We thin \(M_\xi^0\) to control the old coordinates. If \(\xi=0\), let \(M_\xi\) be any infinite subset of \(M_\xi^0\). If \(\xi>0\), the countable product
\[
        P_\xi=\prod_{\eta<\xi}(\omega^{\beta_\eta}+1)
\]
is compact and metrizable. Define a sequence \((x_m)_{m\in M_\xi^0}\) in \(P_\xi\) by
\[
        x_m(\eta)=
        \begin{cases}
        0, & \varphi_\xi(m)\in R_\eta,\\
        q_{h_\eta(\varphi_\xi(m))}^{\eta},
           & \varphi_\xi(m)\in S_\eta,
        \end{cases}
\]
pass to a convergent subsequence, and let \(M_\xi\subseteq M_\xi^0\) be its index set. For every \(\eta<\xi\), the \(\eta\)-coordinate \((x_m(\eta))_{m\in M_\xi}\) converges in \(\omega^{\beta_\eta}+1\), so \(\{x_m(\eta):m\in M_\xi\}\) has at most one limit point in \(\omega^{\beta_\eta}+1\) and therefore belongs to \(\conv_{\beta_\eta}\).

Define
\[
        J_\xi=\varphi_\xi[M_\xi],\qquad R_\xi=R_\xi^0,\qquad S_\xi=\omega\setminus R_\xi.
\]
By Lemma~\ref{lem:positive-indexing}, choose an enumeration \(L_{\beta_\xi}=\{\ell_m^\xi:m<\omega\}\) without repetitions such that \(\{\ell_m^\xi:m\in M_\xi\}\notin\conv_{\beta_\xi}\). Since \(m\in M_\xi\) implies \(\varphi_\xi(m)\in J_\xi\), we have
\[
        \{\ell_m^\xi:m\in M_\xi\}\subseteq\{\ell_m^\xi:\varphi_\xi(m)\in J_\xi\}.
\]
The smaller set is outside \(\conv_{\beta_\xi}\), so by heredity of the ideal the larger set is also outside \(\conv_{\beta_\xi}\). This proves~(i).

For \(0<\xi<\kappa\), fix a surjection \(\rho_\xi:\omega\to\xi\). To define \(h_\xi\), let, for \(C\in\mathcal C_\xi\) and \(n<\omega\),
\[
        K(C,n)=
        \begin{cases}
        C\cap S_\xi, & \xi=0,\\
        (C\cap S_\xi)\setminus\bigcup_{k<n}J_{\rho_\xi(k)}, & \xi>0.
        \end{cases}
\]
This set is infinite. If \(\xi=0\), this follows directly from the choice of \(R_\xi^0\). If \(\xi>0\), then \(C\setminus\bigcup_{k<n}J_{\rho_\xi(k)}\) is itself a member of \(\mathcal C_\xi\), and by the choice of \(R_\xi^0\) it meets \(S_\xi\) in an infinite set. Enumerate the countable family of pairs as \(\{(C_t,n_t):t<\omega\}\) without repetitions. At stage \(s<\omega\), successively for \(t\le s\), choose
\[
        p_{s,t}\in K(C_t,n_t)\setminus\bigl(P_0\cup\cdots\cup P_{s-1}\cup\{p_{s,u}:u<t\}\bigr),
\]
where \(P_r=\{p_{r,t}:t\le r\}\); only finitely many points are excluded from the infinite set \(K(C_t,n_t)\), so the choice is possible. Put \(A(C_t,n_t)=\{p_{s,t}:s\ge t\}\); the sets \(A(C_t,n_t)\) are then pairwise disjoint and infinite. Define \(h_\xi(j)=n\) on each \(A(C,n)\), and \(h_\xi(j)=0\) on the remaining points of \(S_\xi\). Then \(C\cap S_\xi\cap h_\xi^{-1}(\{n\})\supseteq A(C,n)\) is infinite for every \(C\in\mathcal C_\xi\) and every \(n<\omega\). Moreover, if \(\xi>0\), then for every \(k<\omega\) and every \(n>k\), the set \(A(C,n)\) lies in \(K(C,n)\), which excludes \(J_{\rho_\xi(k)}\); hence
\[
        h_\xi[J_{\rho_\xi(k)}\cap S_\xi]\subseteq\{0,1,\ldots,k\}.
\]

It remains to check that the auxiliary invariant continues to hold for \(F',G'\subseteq\xi+1\) finite. Write \(F=F'\setminus\{\xi\}\), \(G=G'\setminus\{\xi\}\), and let \(C\in\mathcal C_\xi\) be the cell corresponding to \(F,G\) and the restriction of \(s'\) to \(F\). If \(\xi\notin F'\), then the new cell is \(C\) (if \(\xi\notin G'\)) or \(C\setminus J_\xi\) (if \(\xi\in G'\)); since \(J_\xi\subseteq R_\xi^0\) and \(C\setminus R_\xi^0\) is infinite, the new cell is infinite. If \(\xi\in F'\) with \(s'(\xi)=n\), then \(A(C,n)\) is contained in \(S_\xi\cap h_\xi^{-1}(\{n\})\cap C\), is infinite, and is disjoint from \(J_\xi\subseteq R_\xi\); so the new cell contains \(A(C,n)\) and is infinite.

We verify (iii) for an arbitrary distinct pair \(\eta,\xi<\kappa\). First suppose \(\eta<\xi\). For \(j\in J_\xi\cap S_\eta\), write \(j=\varphi_\xi(m)\) for some \(m\in M_\xi\); the definition of \(x_m\) gives \(q_{h_\eta(j)}^\eta=x_m(\eta)\). Therefore
\[
        \{q_n^\eta:n\in h_\eta[J_\xi\cap S_\eta]\}\subseteq\{x_m(\eta):m\in M_\xi\}\in\conv_{\beta_\eta},
\]
and the smaller set is in \(\conv_{\beta_\eta}\) by heredity. Now suppose \(\eta>\xi\). At stage \(\eta\), choose \(k<\omega\) with \(\rho_\eta(k)=\xi\). By construction \(h_\eta[J_\xi\cap S_\eta]\subseteq\{0,1,\ldots,k\}\); the corresponding \(q\)-image is finite, hence in \(\conv_{\beta_\eta}\). The reservoir-HMP independence condition (ii) is the special case \(G=\varnothing\) of the auxiliary invariant.
\end{proof}

\section{Proof of the non-reduction}

The proof of the non-reduction uses only maps from the enlarged dense set back to the original one. The first proposition shows that no generality is lost by considering retractions. This removes a minor nuisance in the diagonal argument, because a map on the new points can then be coded by a single element of \(\omega^\omega\).

The failure of a retraction can be detected at a single coordinate. It is enough to find one coordinate \(\beta_\xi\) and one set \(B(J_\xi)\in\Conv(A,D)\) whose inverse image has a non-small \(\beta_\xi\)-projection. The off-diagonal clauses from Proposition~\ref{prop:trace-controlled-selection} ensure that the same set \(B(J_\xi)\) remains small at every other coordinate.

\begin{proposition}
\label{prop:witness-to-retraction}
Let \(D\subseteq E\subseteq X_A\) be countable. If a map \(f\colon E\to D\) witnesses \(\Conv(A,D)\Kle\Conv(A,E)\), then the retraction \(r\colon E\to D\) defined by \(r|_D=\mathrm{id}_D\) and \(r(x)=f(x)\) for \(x\in E\setminus D\) also witnesses \(\Conv(A,D)\Kle\Conv(A,E)\).
\end{proposition}

\begin{proof}
Let \(B\in\Conv(A,D)\). Since \(D\subseteq E\), \(B\) is in particular a subset of \(E\) with the same coordinate projections, so \(B\in\Conv(A,E)\). Then \(f^{-1}[B]\in\Conv(A,E)\) because \(f\) is a witness for \(\Conv(A,D)\Kle\Conv(A,E)\). From the definition of \(r\),
\[
        r^{-1}[B]\subseteq B\cup f^{-1}[B],
\]
and the right-hand side belongs to \(\Conv(A,E)\) as the union of two members of that ideal.
\end{proof}

\begin{proposition}
\label{prop:cross-trace-criterion}
Assume that \(A=\{\beta_\xi:\xi<\omega_1\}\) is an enumeration without repetitions, and let the data \(J_\xi\), \(R_\xi\), \(S_\xi\), \(h_\xi\), \(\ell_m^\xi\) of Proposition~\ref{prop:trace-controlled-selection} be given. Define \(D=\{d_j:j<\omega\}\subseteq X_A\) by
\[
        d_j(\beta_\xi)=
        \begin{cases}
        0, & j\in R_\xi,\\
        q_{h_\xi(j)}^\xi, & j\in S_\xi.
        \end{cases}
\]
For a fixed \(\xi<\omega_1\), put \(B(J_\xi)=\{d_j:j\in J_\xi\}\). Then \(B(J_\xi)\in\Conv(A,D)\). Moreover, if \(E=\{e_m:m<\omega\}\subseteq X_A\) satisfies \(e_m(\beta_\xi)=\ell_m^\xi\), and if a retraction \(r\colon D\cup E\to D\) satisfies
\[
        r(e_m)=d_{\varphi_\xi(m)}\qquad(m<\omega),
\]
then \(r\) is not a \Katetov\ witness for \(\Conv(A,D)\Kle\Conv(A,D\cup E)\).
\end{proposition}

\begin{proof}
We first show \(B(J_\xi)\in\Conv(A,D)\). At the diagonal coordinate \(\beta_\xi\), we have \(J_\xi\subseteq R_\xi\), so \(\pi_{\beta_\xi}[B(J_\xi)]\subseteq\{0\}\in\conv_{\beta_\xi}\). Let \(\eta\neq\xi\); each \(d_j\) with \(j\in J_\xi\) has \(\beta_\eta\)-coordinate \(0\) (if \(j\in R_\eta\)) or \(q_{h_\eta(j)}^\eta\) (if \(j\in S_\eta\)), so
\[
        \pi_{\beta_\eta}[B(J_\xi)]\subseteq\{0\}\cup\{q_n^\eta:n\in h_\eta[J_\xi\cap S_\eta]\}.
\]
The second set is in \(\conv_{\beta_\eta}\) by Proposition~\ref{prop:trace-controlled-selection}(iii), and adjoining the singleton \(\{0\}\) preserves ideal membership. Hence \(\pi_{\beta_\eta}[B(J_\xi)]\in\conv_{\beta_\eta}\) for every \(\eta\neq\xi\), so \(B(J_\xi)\in\Conv(A,D)\).

If \(\varphi_\xi(m)\in J_\xi\), then \(r(e_m)=d_{\varphi_\xi(m)}\in B(J_\xi)\), so \(e_m\in r^{-1}[B(J_\xi)]\), and \(e_m(\beta_\xi)=\ell_m^\xi\) appears in \(\pi_{\beta_\xi}[r^{-1}[B(J_\xi)]]\). Therefore
\[
        \pi_{\beta_\xi}[r^{-1}[B(J_\xi)]]\supseteq\{\ell_m^\xi:\varphi_\xi(m)\in J_\xi\},
\]
and the right-hand side is not in \(\conv_{\beta_\xi}\) by Proposition~\ref{prop:trace-controlled-selection}(i). By heredity, \(\pi_{\beta_\xi}[r^{-1}[B(J_\xi)]]\notin\conv_{\beta_\xi}\), so \(r^{-1}[B(J_\xi)]\notin\Conv(A,D\cup E)\). Since \(B(J_\xi)\in\Conv(A,D)\), the map \(r\) is not a \Katetov\ witness for \(\Conv(A,D)\Kle\Conv(A,D\cup E)\).
\end{proof}

\begin{proof}[Proof of Theorem~\ref{thm:main}]
By CH, enumerate \(\omega^\omega\) as \(\{\varphi_\xi:\xi<\omega_1\}\) and enumerate \(A\) as \(\{\beta_\xi:\xi<\omega_1\}\) without repetitions. For each \(\xi<\omega_1\), enumerate a countable dense subset
\[
        \{q_n^\xi:n<\omega\}\subseteq(\omega^{\beta_\xi}+1)\setminus L_{\beta_\xi};
\]
such a sequence exists because the complement of \(L_{\beta_\xi}\) is dense in \(\omega^{\beta_\xi}+1\) (Lemma~\ref{lem:L-beta}). Apply Proposition~\ref{prop:trace-controlled-selection} with \(\kappa=\omega_1\) to obtain the data \(J_\xi,R_\xi,S_\xi,h_\xi,\ell_m^\xi\). Define \(d_j\in X_A\) by
\[
        d_j(\beta_\xi)=
        \begin{cases}
        0, & j\in R_\xi,\\
        q_{h_\xi(j)}^\xi, & j\in S_\xi,
        \end{cases}
\]
and put \(D=\{d_j:j<\omega\}\).

\smallskip
\emph{Density of \(D\).} Let \(U\subseteq X_A\) be a non-empty basic open set, determined by a finite \(F\subseteq\omega_1\) and non-empty open sets \(U_\xi\subseteq\omega^{\beta_\xi}+1\) for \(\xi\in F\); the remaining coordinates of \(U\) range over the full coordinate spaces and impose no constraint. For each \(\xi\in F\), the set \(\{q_n^\xi:n<\omega\}\) is dense in \(\omega^{\beta_\xi}+1\), so we may choose \(n_\xi<\omega\) with \(q_{n_\xi}^\xi\in U_\xi\). By Proposition~\ref{prop:trace-controlled-selection}(ii), there is \(j\in\bigcap_{\xi\in F}S_\xi\) with \(h_\xi(j)=n_\xi\) for every \(\xi\in F\). Then \(d_j(\beta_\xi)=q_{n_\xi}^\xi\in U_\xi\) for every \(\xi\in F\), so \(d_j\in U\). Hence \(D\) meets every non-empty basic open subset of \(X_A\).

\smallskip
\emph{The enlarged dense set.} For \(m<\omega\) define \(e_m\in X_A\) by
\[
        e_m(\beta_\xi)=\ell_m^\xi\qquad(\xi<\omega_1),
\]
and put \(E=\{e_m:m<\omega\}\) and \(D^*=D\cup E\). At every coordinate \(\beta_\xi\), every \(d_j\) lies outside \(L_{\beta_\xi}\): the value \(0\) is not of the form \(\omega^2\cdot n+\omega\cdot(k+1)\) with \(n\ge 1\), and each \(q_n^\xi\) lies outside \(L_{\beta_\xi}\) by construction. On the other hand, \(e_m(\beta_\xi)=\ell_m^\xi\in L_{\beta_\xi}\) for every \(\xi\). Hence \(D\cap E=\varnothing\), and \(D^*\) is countable and dense (since it contains \(D\)).

\smallskip
\emph{Easy direction.} The inclusion \(i\colon D\hookrightarrow D^*\) witnesses \(\Conv(A,D^*)\Kle\Conv(A,D)\): for \(B\in\Conv(A,D^*)\), the set \(i^{-1}[B]=B\cap D\) projects coordinatewise to subsets of the corresponding \(\pi_\alpha[B]\in\conv_\alpha\), so \(B\cap D\in\Conv(A,D)\).

\smallskip
\emph{Hard direction.} Assume \(\Conv(A,D)\Kle\Conv(A,D^*)\); we derive a contradiction. By Proposition~\ref{prop:witness-to-retraction}, this reduction is witnessed by a retraction \(r\colon D^*\to D\). For each \(m<\omega\), the point \(r(e_m)\) belongs to \(D\), and we fix any \(\psi(m)<\omega\) with \(r(e_m)=d_{\psi(m)}\). Any such choice gives an element \(\psi\in\omega^\omega\), so \(\psi=\varphi_\xi\) for some \(\xi<\omega_1\). Proposition~\ref{prop:cross-trace-criterion} then gives a set \(B(J_\xi)\in\Conv(A,D)\) with
\[
        r^{-1}[B(J_\xi)]\notin\Conv(A,D^*),
\]
contradicting the assumption that \(r\) is a \Katetov\ witness, and therefore \(\Conv(A,D)\not\Kle\Conv(A,D^*)\).
\end{proof}

\section{Open questions}

The CH assumption enters the construction through the enumeration of \(\omega^\omega\) by \(\omega_1\). Since the available coordinates are indexed by a subset of \([3,\omega_1)\), the diagonal argument assigns one coordinate to each possible code only under CH. A ZFC proof would require a different way to handle all possible witnesses.

The ZFC absorption result of Section~\ref{sec:zfc-stability} rules out dense-set dependence examples in which the two directional non-small-coordinate sets are both countable. Under CH, Theorem~\ref{thm:main} gives an example in which the obstruction appears on uncountably many coordinates. The following two problems remain open.

\begin{question}
Does the conclusion of Theorem~\ref{thm:main} hold in ZFC?
\end{question}

\begin{question}\label{q:GLB}
Does the family \(\{\conv_\alpha:2\le\alpha<\omega_1\}\) have a greatest lower bound in the \Katetov\ order?
\end{question}

The two questions are related but not identical. Theorem~\ref{thm:main} rules out one simple way of obtaining a canonical uncountable-coordinate candidate from countable dense subsets of \(X_A\). It does not exclude the possibility that a different ideal, not tied to a particular dense set in such a product, is the greatest lower bound of the family \(\{\conv_\alpha:2\le\alpha<\omega_1\}\).

\end{document}